\documentclass[letterpaper, 10 pt, conference]{ieeeconf}  

\IEEEoverridecommandlockouts                              
\usepackage{graphicx}
\usepackage{amsmath}
\usepackage{amsfonts}
\usepackage{amssymb}
\usepackage{subcaption}
\usepackage{float}
\usepackage{algpseudocode}

\usepackage{multirow}
\usepackage{enumerate}
\usepackage{tikz}

\usepackage{pgfplots}   

\usepackage{booktabs}   
\usepackage{multirow}   

\usepackage{float}      

\usepackage{booktabs}
\usepackage{algorithm}

\usepackage{xcolor}
\usepackage{mathtools}
\usepackage{acronym}
\usepackage{bm}         
\usepackage{hyperref}

\usepackage{amsthm}
\theoremstyle{plain} 
\newtheorem{theorem}{Theorem}[section]
\newtheorem{proposition}[theorem]{Proposition}
\newtheorem{corollary}[theorem]{Corollary}
\theoremstyle{definition}
\newtheorem{definition}[theorem]{Definition}
\theoremstyle{remark} 
\newtheorem{remark}[theorem]{Remark}
\theoremstyle{definition}
\newtheorem{example}{Example}
\theoremstyle{lemma}
\newtheorem{lemma}[theorem]{Lemma}

\usepackage{flushend}

\begin{document}
\title{Roundabout Constrained Convex Generators: A Unified Framework \\ for Multiply-Connected Reachable Sets}

\author{Peng Xie, Sabin Diaconescu, Florin Stoican, Amr Alanwar
\thanks{
Peng Xie, Amr Alanwar are with the TUM School of Computation, Information and Technology, Department of Computer Engineering, Technical University of Munich, Heilbronn, Germany. 
Email: \{p.xie, alanwar\}@tum.de\newline
Sabin Diaconescu, Florin Stoican are with the Nat. Univ. of Science and Technology Politehnica Bucharest, Romania.
Email: \{sabin.diaconescu@stud.acs,florin.stoican@\}.upb.ro
}
}

\maketitle

\begin{abstract}
This paper introduces Roundabout Constrained Convex Generators (RCGs), a set representation framework for modeling multiply connected regions in control and verification applications. The RCG representation extends the constrained convex generators framework by incorporating an inner exclusion zone, creating sets with topological holes that naturally arise in collision avoidance and safety-critical control problems. We present two equivalent formulations: a set difference representation that provides geometric intuition and a unified parametric representation that facilitates computational implementation. The paper establishes closure properties under fundamental operations, including linear transformations, Minkowski sums, and intersections with convex generator sets. We derive special cases, including roundabout zonotopes and roundabout ellipsotopes, which offer computational advantages for specific norm selections. The framework maintains compatibility with existing optimization solvers while enabling the representation of non-convex feasible regions that were previously challenging to model efficiently.
\end{abstract}

\section{Introduction} In the domain of safety-critical autonomous systems, such as autonomous driving, robot navigation, and unmanned aerial vehicle (UAV) swarm control, the precise and efficient modeling of system states is a cornerstone for achieving reliable operation. Many real-world tasks not only require the system to remain within a specific safe region but also demand that it actively avoids a series of ``forbidden zones'' or obstacles~\cite{girard2005reachability}. These constraints collectively define a complex, often \emph{non-convex}, and even \emph{multiply connected} feasible state space. Devising a mathematical representation for such complex sets with ``holes'' that is compact, efficient, and compatible with modern optimization algorithms remains a long-standing challenge in control theory and formal verification.

A variety of methods exist for representing non-convex sets. For instance, approaches like constrained polynomial zonotopes~\cite{kochdumper2023constrained} or hybrid systems (e.g., hybrid zonotopes~\cite{bird2023hybrid}) offer powerful tools for capturing complex, non-convex geometries. Similarly, unions of multiple convex sets (such as constrained zonotopes~\cite{scott2016constrained} or ellipsotopes~\cite{kousik2022ellipsotopes}) can approximate some non-convex shapes~\cite{xie2025hybrid}. However, while these formalisms are adept at handling general non-convexity, they often lack a dedicated and compact structure to efficiently represent \textit{multiply connected topologies}~\cite{walsh1956conformal}. Approximating a region with a hole might require a large number of convex sets in a union~\cite{bird2021unions}, or a high-degree polynomial, which can lead to a combinatorial explosion in computational complexity during subsequent analysis and synthesis~\cite{althoff2013reachability,althoff2010reachability,alanwar2023data,alanwar2022data,xie2025data}. 

To address this gap, we introduce the \textit{Roundabout Constrained Convex Generators (RCG)}, a set representation whose geometry mirrors a traffic roundabout: a navigable outer boundary with an inaccessible central region. As in Fig.~\ref{fig:roundabout_traffic}, admissible trajectories must circulate within the outer domain while avoiding the inner exclusion zone; such topology arises in collision avoidance, robust control with keep-out regions, and motion planning problems where specific areas must be circumnavigated~\cite{xie2025informed}.  Throughout this work, we restrict attention to feasible regions in arbitrary ambient dimension that contain exactly one topological hole. In the planar case ($\mathbb{R}^2$), such a set is a \emph{doubly connected domain}. The RCG’s \emph{outer} feasible domain is encoded by \emph{Constrained Convex Generators (CCG)}~\cite{silverccg2022}, preserving solver-friendly structure (e.g., compatibility with affine maps and standard convex--algebraic operations~\cite{silvestre2023exact,kousik2022ellipsotopes}), while an \emph{inner} exclusion constraint carves out the characteristic hole. Looking ahead, we will extend this concept to a \emph{hybrid} setting---i.e., hybrid roundabout constrained convex generators---to systematically address cases with $k$ holes, quantified via Betti numbers~\cite{ghrist2008barcodes}, thereby accommodating unions of obstacles that may partially overlap.
\begin{figure}[h]
\centering
\includegraphics[width=0.45\textwidth]{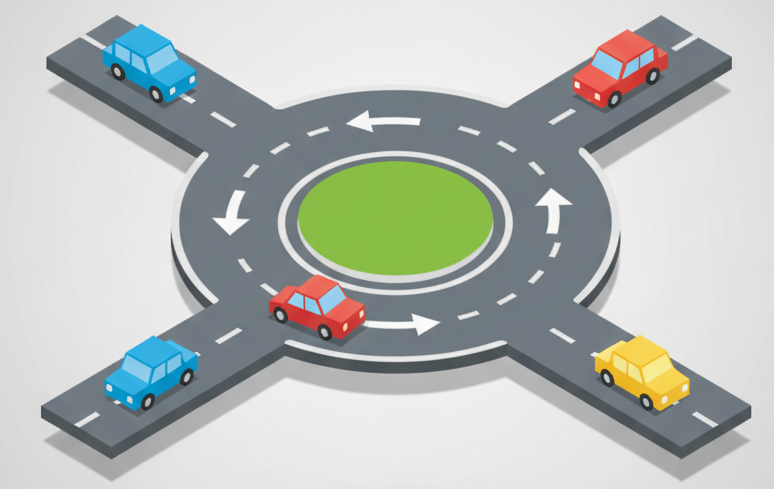}
\caption{A traffic roundabout: admissible motion (gray) circulates around an inaccessible island (green).}
\label{fig:roundabout_traffic}
\end{figure}


This unified construction yields a compact, dimension-agnostic model of annular/ring-like feasible regions that generalizes beyond spherical/ellipsoidal keep-outs. In this paper, we formally define RCGs, show that intersections of RCGs remain within the class, and demonstrate how these properties enable systematic reachability analysis and formal verification.

 The main contributions of this paper are: (i) we formally propose and define the RCG as a closed-form representation for multiply connected feasible regions; (ii) we derive algebraic properties under fundamental set operations and establish closure results (e.g., intersection), laying the groundwork for controller synthesis and reachability analysis.  We release an open-source implementation at
  \url{git@github.com:TUM-CPS-HN/RCG.git}.

The remainder of this paper is organized as follows: Section~II presents the basic preliminaries used throughout the paper. Section~III details the mathematical definition and properties of the RCG. Section~IV develops basic set operations for the RCG, including linear transformations, Minkowski sums, and intersections. Section~V concludes the paper.

\section{Preliminaries}
This section introduces the mathematical preliminaries that underpin the proposed RCG framework.

\subsection{Notation}

We denote vectors by bold lowercase letters, e.g., $\mathbf{x}, \mathbf{c} \in \mathbb{R}^n$, and matrices by uppercase letters, e.g., $\mathbf{G} \in \mathbb{R}^{n \times m}$. For Greek letters representing vectors, we use slanted symbols, e.g., $\bm{\beta} \in \mathbb{R}^m$. For $p \in [1, \infty)$, the $p$-norm is defined as $\|\mathbf{x}\|_p = \left(\sum_{i=1}^n |x_i|^p\right)^{1/p}$, with the standard extension $\|\mathbf{x}\|_\infty = \max_i |x_i|$. The Minkowski sum of two sets $\mathcal{X}, \mathcal{Y} \subseteq \mathbb{R}^n$ is defined as $\mathcal{X} \oplus \mathcal{Y} = \{ \mathbf{x} + \mathbf{y} \mid \mathbf{x} \in \mathcal{X}, \mathbf{y} \in \mathcal{Y} \}$.

\subsection{Constrained Convex Generators}

We begin by defining the Constrained Convex Generator (CCG), a set representation that forms the basis of our proposed framework.

\begin{definition}[Constrained Convex Generator \cite{silverccg2022}]
\label{def:ccg}
A Constrained Convex Generator ($\mathcal{CG}$) is a set defined as:
\begin{equation}
\label{eq:ccg_def}
\mathcal{CG} = \left\{ \mathbf{c} + \mathbf{G}\bm{\beta} \;\Big|\; \|\bm{\beta}_{\mathcal{J}_i}\|_{p_i} \leq 1, \forall i \in \{1,\ldots,k\}, \mathbf{A}\bm{\beta} = \mathbf{b} \right\},
\end{equation}
where:
\begin{itemize}
\item $\mathbf{c} \in \mathbb{R}^n$ is the center of the set
\item $\mathbf{G} \in \mathbb{R}^{n \times m}$ is the generator matrix, with columns $\mathbf{g}_j$
\item $\bm{\beta} \in \mathbb{R}^m$ is the coefficient vector
\item $\mathcal{J} = \{\mathcal{J}_1, \mathcal{J}_2, \ldots, \mathcal{J}_k\}$ is a partition of the index set $\{1, 2, \ldots, m\}$
\item $\mathcal{P} = \{p_1, p_2, \ldots, p_k\}$ with $p_i \in [1, \infty]$ are the norms corresponding to each partition
\item $\mathbf{A} \in \mathbb{R}^{q \times m}$ and $\mathbf{b} \in \mathbb{R}^q$ define a set of linear constraints on the coefficients
\end{itemize}
For brevity, we use the tuple $\langle \mathbf{c}, \mathbf{G}, \mathcal{J}, \mathcal{P}, \mathbf{A}, \mathbf{b} \rangle$ to denote a $\mathcal{CG}$.
\end{definition}

CCGs are closed under several important operations, making them suitable for reachability analysis.

\begin{proposition}[Fundamental Operations \cite{silvestre2023exact}]
\label{prop:fundamental_ops}
Let $\mathcal{CG}_1 = \langle \mathbf{c}_1, \mathbf{G}_1, \mathcal{J}_1, \mathcal{P}_1, \mathbf{A}_1, \mathbf{b}_1 \rangle$ and $\mathcal{CG}_2 = \langle \mathbf{c}_2, \mathbf{G}_2, \mathcal{J}_2, \mathcal{P}_2, \mathbf{A}_2, \mathbf{b}_2 \rangle$ be two CCGs in $\mathbb{R}^n$, where $\mathbf{G}_1 \in \mathbb{R}^{n \times m_1}$.

\begin{enumerate}
\item \textit{Minkowski Sum:} Their Minkowski sum $\mathcal{CG}_1 \oplus \mathcal{CG}_2$ is a CCG given by $\langle \mathbf{c}, \mathbf{G}, \mathcal{J}, \mathcal{P}, \mathbf{A}, \mathbf{b} \rangle$, where:
\begin{align}
\mathbf{c} &= \mathbf{c}_1 + \mathbf{c}_2, \quad 
\mathbf{G} = [\mathbf{G}_1 \quad \mathbf{G}_2], \\
\mathbf{A} &= \begin{bmatrix} 
\mathbf{A}_1 & \mathbf{0} \\ 
\mathbf{0} & \mathbf{A}_2 
\end{bmatrix}, \quad 
\mathbf{b} = \begin{bmatrix} 
\mathbf{b}_1 \\ 
\mathbf{b}_2 
\end{bmatrix},
\end{align}
and $\mathcal{J}, \mathcal{P}$ are the appropriately concatenated and index-shifted partitions and norm sets.

\item \textit{Linear Transformation:} For a matrix $\mathbf{T} \in \mathbb{R}^{p \times n}$, the set $\mathbf{T}(\mathcal{CG}_1)$ is a CCG given by:
\begin{equation}
\mathbf{T}(\mathcal{CG}_1) = \langle \mathbf{T}\mathbf{c}_1, \mathbf{T}\mathbf{G}_1, \mathcal{J}_1, \mathcal{P}_1, \mathbf{A}_1, \mathbf{b}_1 \rangle.
\end{equation}
\end{enumerate}
\end{proposition}

As reviewed in~\cite{silverccg2022}, CCGs admit exact, closed-form representations for \emph{Minkowski sums}, \emph{affine/linear transformations}, and \emph{intersections}. Building on these standard operations, we now provide a new closed-form construction, for the intersection of a CCG with a halfspace.

\begin{proposition}[Halfspace Intersection]
\label{prop:halfspace}
Let $\mathcal{CG} = \langle \mathbf{c}, \mathbf{G}, \mathcal{J}, \mathcal{P}, \mathbf{A}, \mathbf{b} \rangle$ be a CCG in $\mathbb{R}^n$ with $m$ generators. Let $\mathcal{H} = \{\mathbf{x} \in \mathbb{R}^n : \mathbf{h}^T\mathbf{x} \leq f\}$ be a halfspace. Define the scalar representing the maximum slack:
\begin{equation}
\label{eq:dmax}
d_{\max} \;=\; f - \mathbf{h}^T\mathbf{c} + \sum_{j=1}^{m} \big|\mathbf{h}^T\mathbf{g}_j\big|.
\end{equation}
Then:
\begin{enumerate}
\item The intersection $\mathcal{CG} \cap \mathcal{H}$ is empty if and only if $d_{\max} < 0$.

\item If $d_{\max} \geq 0$, the intersection is exactly represented by a new CCG, $\mathcal{CG}' = \langle \mathbf{c}, \mathbf{G}', \mathcal{J}', \mathcal{P}', \mathbf{A}', \mathbf{b}' \rangle$, where:
\begin{align}
\label{eq:Gprime}
\mathbf{G}' &= [\mathbf{G} \;\; \mathbf{0}_{n \times 1}] \in \mathbb{R}^{n \times (m+1)} ,\\
\label{eq:JprimePprime}
\mathcal{J}' &= \mathcal{J} \cup \{m{+}1\}, \quad \mathcal{P}' = \mathcal{P} \cup \{\infty\}, \\
\label{eq:Aprime}
\mathbf{A}' &= \begin{bmatrix} 
\mathbf{A} & \mathbf{0} \\ 
\mathbf{h}^T\mathbf{G} & d_{\max}/2 
\end{bmatrix} \in \mathbb{R}^{(q+1) \times (m+1)} ,\\
\label{eq:bprime}
\mathbf{b}' &= \begin{bmatrix} 
\mathbf{b} \\ 
f - \mathbf{h}^T\mathbf{c} - d_{\max}/2 
\end{bmatrix} \in \mathbb{R}^{q+1}.
\end{align}
Equivalently, the augmented linear constraint for $\widetilde{\bm{\beta}}$ reads
\begin{equation}
\label{eq:aug_constraint}
\mathbf{A}'\,\widetilde{\bm{\beta}} \;=\; \mathbf{b}'.
\end{equation}
\end{enumerate}
\end{proposition}

\begin{proof}
We establish exactness by proving both set inclusions. Let $\widetilde{\bm{\beta}} = [\bm{\beta}^T \;\; \beta_{m+1}]^T$ denote the augmented coefficient vector for $\mathcal{CG}'$.

\textit{Forward inclusion}:
Let $\mathbf{x} \in \mathcal{CG} \cap \mathcal{H}$. Since $\mathbf{x} \in \mathcal{CG}$, there exists $\bm{\beta} \in \mathbb{R}^m$ such that $\mathbf{x} = \mathbf{c} + \mathbf{G}\bm{\beta}$ and all original constraints on $\bm{\beta}$ hold. Because $\mathbf{x} \in \mathcal{H}$, we have $\mathbf{h}^T\mathbf{x} \leq f$. Define
\begin{equation}
\label{eq:eps_def}
\varepsilon \;=\; f - \mathbf{h}^T\mathbf{x} \;\ge 0 .
\end{equation}
By the triangle inequality and the definition \eqref{eq:dmax}, one obtains
\begin{equation}
\label{eq:eps_bound}
0 \;\le\; \varepsilon \;\le\; d_{\max}.
\end{equation}

If $d_{\max}=0$, then by \eqref{eq:eps_def} we have $\varepsilon=0$ and $\mathbf{h}^T\mathbf{x}=f$; any $\beta_{m+1}\in[-1,1]$ satisfies \eqref{eq:aug_constraint}. If $d_{\max}>0$, set
\begin{equation}
\label{eq:beta_m1_choice}
\beta_{m+1}\;=\; \frac{2\varepsilon}{d_{\max}} - 1 \;\in\; [-1,1].
\end{equation}
Then, substituting \eqref{eq:eps_def} and \eqref{eq:beta_m1_choice} into the last row of \eqref{eq:aug_constraint} (with \eqref{eq:Aprime}--\eqref{eq:bprime}), we obtain
\begin{align}
\label{eq:last_row_eval}
\mathbf{h}^T\mathbf{G}\bm{\beta} + \frac{d_{\max}}{2}\beta_{m+1}
&= \big(\mathbf{h}^T\mathbf{x} - \mathbf{h}^T\mathbf{c}\big) + \frac{d_{\max}}{2}\!\left(\frac{2\varepsilon}{d_{\max}} - 1\right) \nonumber\\
&= \big(f - \varepsilon - \mathbf{h}^T\mathbf{c}\big) + \varepsilon - \frac{d_{\max}}{2} \nonumber\\
&= f - \mathbf{h}^T\mathbf{c} - \frac{d_{\max}}{2},
\end{align}
which is exactly the last component of $\mathbf{b}'$ in \eqref{eq:bprime}. Hence \eqref{eq:aug_constraint} holds; together with \eqref{eq:Gprime} (whose last column is zero), this yields $\mathbf{x}=\mathbf{c}+\mathbf{G}'\widetilde{\bm{\beta}}$, i.e., $\mathbf{x}\in\mathcal{CG}'$.

\textit{Reverse inclusion}:
Let $\mathbf{x} \in \mathcal{CG}'$. Then by \eqref{eq:Gprime} we have $\mathbf{x} = \mathbf{c} + \mathbf{G}\bm{\beta}$ for some $\widetilde{\bm{\beta}}$ satisfying \eqref{eq:aug_constraint}. The original constraints on $\bm{\beta}$ hold by construction of \eqref{eq:Aprime}, hence $\mathbf{x}\in\mathcal{CG}$. For the halfspace, using the last row of \eqref{eq:aug_constraint} and \eqref{eq:bprime} gives
\begin{align}
\label{eq:hx_final}
\mathbf{h}^T\mathbf{x} &= \mathbf{h}^T\mathbf{c} + \mathbf{h}^T\mathbf{G}\bm{\beta} \nonumber\\
&= \mathbf{h}^T\mathbf{c} + \Big(f - \mathbf{h}^T\mathbf{c} - \frac{d_{\max}}{2}\Big) - \frac{d_{\max}}{2}\beta_{m+1} \nonumber\\
&= f - \frac{d_{\max}}{2}\big(1 + \beta_{m+1}\big) \;\le\; f,
\end{align}
since $d_{\max}\ge 0$ and $|\beta_{m+1}|\le 1$. Thus $\mathbf{x}\in\mathcal{H}$.

\textit{Emptiness condition:} 
If $d_{\max}<0$, then from \eqref{eq:dmax} we have $f < \mathbf{h}^T\mathbf{c} - \sum_{j=1}^m |\mathbf{h}^T\mathbf{g}_j|$, which is a valid lower bound on $\min_{\mathbf{x}\in\mathcal{CG}} \mathbf{h}^T\mathbf{x}$ (via the support-function bound for generator representations). Hence $\min_{\mathbf{x}\in\mathcal{CG}} \mathbf{h}^T\mathbf{x} > f$ and the intersection is empty. Conversely, if $d_{\max}\ge 0$, the construction above shows feasibility of \eqref{eq:aug_constraint}, hence non-emptiness.
\end{proof}

\section{Roundabout Constrained Convex Generators }

The set representation roundabout constrained convex generators derives its name from its geometric resemblance to a traffic roundabout, featuring a navigable outer boundary with an inaccessible central region. To establish the mathematical foundation, we begin with an intuitive representation of the roundabout constrained convex generators as a set difference. This formulation provides conceptual clarity and motivates the more computationally tractable representations that follow.

\begin{definition}[Roundabout Constrained Convex Generators]
\label{def:roundabout_ccg}
A Roundabout Constrained Convex Generators (RCG) set $\mathcal{R}$ is defined as the set difference between an outer CCG $\langle \mathbf{c}_o, \mathbf{G}_o, \mathcal{J}_o, \mathcal{P}_o, \mathbf{A}_o, \mathbf{b}_o \rangle$ and an inner CCG $\langle \mathbf{c}_{in}, \mathbf{G}_{in}, \mathcal{J}_{in}, \mathcal{P}_{in}, \mathbf{A}_{in}, \mathbf{b}_{in} \rangle$:
\begin{equation}\label{set_form}
\mathcal{R} = \mathcal{CG}_{\text{outer}}  \setminus \mathcal{CG}_{\text{inner}},
\end{equation}
which admits the following unified representation:


\begin{equation}\label{unified}
\begin{aligned}
\mathcal{R}
= \Bigl\{\,
& \mathbf{c} + \mathbf{G}\boldsymbol{\beta} \ \Big| \
\|\boldsymbol{\beta}_{\mathcal{J}_i}\|_{p_i} \le 1,\
\forall i = 1,\ldots,k_o, \\
& \mathbf{A}\boldsymbol{\beta} = \mathbf{b},\
\nexists\, \boldsymbol{\eta}:
\mathbf{H}\boldsymbol{\eta} = \mathbf{h}(\boldsymbol{\beta}) \text{ with } \\
& \|\boldsymbol{\eta}_{\mathcal{K}_j}\|_{q_j} \le r_j,\
\forall j = 1,\ldots,k_{in},\
\mathbf{A}_{in}\boldsymbol{\eta} = \mathbf{b}_{in}
\Bigr\}, \\
\text{where } &
\mathbf{c} = \mathbf{c}_o,\quad
\mathbf{G} = \mathbf{G}_o, \text{outer CCG center and generators}, \\
& \mathbf{A} = \mathbf{A}_o,\quad
\mathbf{b} = \mathbf{b}_o, \text{outer linear constraints}, \\
& \mathbf{H} = \mathbf{G}_{in},  \text{inner generator matrix}, \\
& \mathbf{h}(\boldsymbol{\beta})
= (\mathbf{c}_o - \mathbf{c}_{in})
+ \mathbf{G}_o \boldsymbol{\beta},  \text{transformation mapping}, \\
& \mathcal{P}_o = \{p_1, p_2, \ldots, p_{k_o}\}
\text{ with } p_i \in [1,\infty], \\
& \text{norms for outer CCG partitions}, \\
& \mathcal{P}_{in} = \{q_1, q_2, \ldots, q_{k_{in}}\}
\text{ with } q_j \in [1,\infty], \\
& \text{norms for inner CCG partitions}, \\
& r_j \in [0,1),\ \forall j = 1,\ldots,k_{in}, \mathbf{r} = [r_1, r_2, \ldots, r_{k_{in}}],\\
& \text{scaling factor for the inner CCG}.
\end{aligned}
\end{equation}

For brevity, we introduce a standardized tuple representation that encodes all parameters of the roundabout constrained convex generators:
\begin{equation}
\mathcal{R} = \langle \mathbf{G}_o, \mathbf{G}_{in}, \mathbf{c}_o, \mathbf{c}_{in}, \mathcal{P}_o, \mathcal{P}_{in}, \mathbf{r}, \mathbf{A}_o, \mathbf{A}_{in}, \mathbf{b}_o, \mathbf{b}_{in} \rangle.
\end{equation}
\end{definition}

\begin{lemma}[Derivation of Unified RCG Representation]
\label{lem:rccg_derivation}
The set difference representation in \eqref{set_form} can be equivalently expressed as the unified representation in \eqref{unified}.
\end{lemma}

\begin{proof}
We begin by explicitly defining the outer and inner CCG sets that constitute the set difference. The outer CCG is given by:
\begin{equation}
\mathcal{CG}_{\text{outer}} = \left\{ \mathbf{c}_o + \mathbf{G}_o \bm{\beta}_o \;\middle|\;
\begin{aligned}
&\|\bm{\beta}_{o,\mathcal{J}_i}\|_{p_i} \leq 1, \forall i = 1,\ldots,k_o, \\
&\mathbf{A}_o \bm{\beta}_o = \mathbf{b}_o
\end{aligned}
\right\},
\end{equation}
and the inner CCG is defined as:
\begin{equation}
\mkern-4mu\mathcal{CG}_{\text{inner}}\mkern-4mu =\mkern-4mu \left\{\mkern-4mu \mathbf{c}_{in}\mkern-4mu +\mkern-4mu \mathbf{G}_{in} \bm{\beta}_{in}\mkern-4mu \;\middle|\;
\begin{aligned}
&\|\bm{\beta}_{in,\mathcal{K}_j}\|_{q_j} \leq r_j, \forall j = 1,\ldots,k_{in},\mkern-4mu \\
&\mathbf{A}_{in} \bm{\beta}_{in} = \mathbf{b}_{in}
\end{aligned}
\right\}\mkern-4mu.
\end{equation}

To derive the unified representation, we characterize the set difference $\mathcal{R} = \mathcal{CG}_{\text{outer}} \setminus \mathcal{CG}_{\text{inner}}$. A point $\mathbf{x} \in \mathcal{R}$ if and only if $\mathbf{x} \in \mathcal{CG}_{\text{outer}}$ and $\mathbf{x} \notin \mathcal{CG}_{\text{inner}}$.

For the membership condition $\mathbf{x} \in \mathcal{CG}_{\text{outer}}$, there exists a parameter vector $\bm{\beta}$ such that:
\begin{align}
\mathbf{x} &= \mathbf{c}_o + \mathbf{G}_o\bm{\beta},\nonumber \\
\|\bm{\beta}_{\mathcal{J}_i}\|_{p_i} &\leq 1,\quad \forall i = 1,\ldots,k_o,\nonumber \\
\mathbf{A}_o\bm{\beta} &= \mathbf{b}_o.
\end{align}

For the exclusion condition $\mathbf{x} \notin \mathcal{CG}_{\text{inner}}$, there does not exist any parameter vector $\bm{\eta}$ that would satisfy:
\begin{align}
\mathbf{x} &= \mathbf{c}_{in} + \mathbf{G}_{in}\bm{\eta}, \nonumber\\
\|\bm{\eta}_{\mathcal{K}_j}\|_{q_j} &\leq r_j,\quad \forall j = 1,\ldots,k_{in},\nonumber \\
\mathbf{A}_{in}\bm{\eta} &= \mathbf{b}_{in}.
\end{align}

Substituting $\mathbf{x} = \mathbf{c}_o + \mathbf{G}_o\bm{\beta}$ into the exclusion condition, we require that no $\bm{\eta}$ exists such that:
\begin{equation}
\mathbf{c}_o + \mathbf{G}_o\bm{\beta} = \mathbf{c}_{in} + \mathbf{G}_{in}\bm{\eta}.
\end{equation}

Rearranging this equation yields:
\begin{equation}
\mathbf{G}_{in}\bm{\eta} = (\mathbf{c}_o - \mathbf{c}_{in}) + \mathbf{G}_o\bm{\beta}.
\end{equation}

By setting $\mathbf{c} = \mathbf{c}_o$, $\mathbf{G} = \mathbf{G}_o$, $\mathbf{A} = \mathbf{A}_o$, $\mathbf{b} = \mathbf{b}_o$, $\mathbf{H} = \mathbf{G}_{in}$, and $\mathbf{h}(\bm{\beta}) = (\mathbf{c}_o - \mathbf{c}_{in}) + \mathbf{G}_o\bm{\beta}$, we obtain the unified representation in \eqref{unified}. 
\end{proof}

\section{Basic Set Operations of RCG}
Having established the RCG representation, we begin with its behavior under linear transformations—the basic mechanism for change of coordinates and dynamics propagation, and a building block for later operations such as Minkowski sum and intersection.

\subsection{Linear Transformation of RCG}

\begin{proposition}[Linear Transformation of RCG]
\label{prop:rccg_linear_transform}
For a matrix $\mathbf{T} \in \mathbb{R}^{m \times n}$ and an RCG $\mathcal{R} = \langle \mathbf{G}_o, \mathbf{G}_{in}, \mathbf{c}_o, \mathbf{c}_{in}, \mathcal{P}_o, \mathcal{P}_{in}, \mathbf{r}, \mathbf{A}_o, \mathbf{A}_{in}, \mathbf{b}_o, \mathbf{b}_{in} \rangle$, the linear transformation yields:
\begin{small}
\begin{equation}
\mathbf{T}(\mathcal{R}) = \langle \mathbf{T}\mathbf{G}_o, \mathbf{T}\mathbf{G}_{in}, \mathbf{T}\mathbf{c}_o, \mathbf{T}\mathbf{c}_{in}, \mathcal{P}_o, \mathcal{P}_{in}, \mathbf{r}, \mathbf{A}_o, \mathbf{A}_{in}, \mathbf{b}_o, \mathbf{b}_{in} \rangle.
\end{equation}
\end{small}

\end{proposition}

\begin{proof}
Since $\mathcal{R} = \mathcal{R}_{\text{outer}} \setminus \mathcal{R}_{\text{inner}}$, applying the linear transformation:
\begin{equation}
\mathbf{T}(\mathcal{R}) = \mathbf{T}(\mathcal{R}_{\text{outer}}) \setminus \mathbf{T}(\mathcal{R}_{\text{inner}}).
\end{equation}

The transformed outer and inner CCGs become:
\begin{align}
\mathbf{T}(\mathcal{R}_{\text{outer}}) &= \{ \mathbf{T}\mathbf{c}_o + \mathbf{T}\mathbf{G}_o\bm{\beta} \mid \text{constraints on } \bm{\beta} \}, \\
\mathbf{T}(\mathcal{R}_{\text{inner}}) &= \{ \mathbf{T}\mathbf{c}_{in} + \mathbf{T}\mathbf{G}_{in}\bm{\eta} \mid \text{constraints on } \bm{\eta} \}.
\end{align}

Parameter constraints remain unchanged as they are defined in parameter space. The resulting structure maintains the RCG form with transformed centers and generators.
\end{proof}

\subsection{Minkowski Sum}

\begin{proposition}[Minkowski Sum of RCG and CCG]
\label{prop:rccg_minkowski_sum}
For an RCG $\mathcal{R} = \langle \mathbf{G}_o, \mathbf{G}_{in}, \mathbf{c}_o, \mathbf{c}_{in}, \mathcal{P}_o, \mathcal{P}_{in}, \mathbf{r}, \mathbf{A}_o, \mathbf{A}_{in}, \mathbf{b}_o, \mathbf{b}_{in} \rangle$ and a CCG $\mathcal{CG} = \langle \mathbf{c}_{\text{ccg}}, \mathbf{G}_{\text{ccg}}, \mathcal{J}_{\text{ccg}}, \mathcal{P}_{\text{ccg}}, \mathbf{A}_{\text{ccg}}, \mathbf{b}_{\text{ccg}} \rangle$, their Minkowski sum yields:
\begin{equation}
\begin{aligned}
\mathcal{R} \oplus \mathcal{CG} = \Bigl\{ &(\mathbf{c}_o + \mathbf{c}_{\text{ccg}}) + [\mathbf{G}_o, \mathbf{G}_{\text{ccg}}]\tilde{\bm{\beta}} \;\Big|\; \\
&\|\tilde{\bm{\beta}}_{\mathcal{J}_i}\|_{p_i} \leq 1, \; i = 1,\ldots,k_o, \\
&\|\tilde{\bm{\beta}}_{\mathcal{L}_j}\|_{p_j} \leq 1, \; j = 1,\ldots,m_{\text{ccg}}, \\
&\begin{bmatrix} \mathbf{A}_o & \mathbf{0} \\ \mathbf{0} & \mathbf{A}_{\text{ccg}} \end{bmatrix}\tilde{\bm{\beta}} = \begin{bmatrix} \mathbf{b}_o \\ \mathbf{b}_{\text{ccg}} \end{bmatrix}, \\
&\nexists \bm{\eta}: \mathbf{G}_{in}\bm{\eta} = (\mathbf{c}_o - \mathbf{c}_{in}) + \mathbf{G}_o\bm{\beta} \\
&\quad\quad + \mathbf{c}_{\text{ccg}} + \mathbf{G}_{\text{ccg}}\bm{\gamma} \\
&\|\bm{\eta}_{\mathcal{K}_l}\|_{q_l} \leq r_l, \; l = 1,\ldots,k_{in}, \; \mathbf{A}_{in}\bm{\eta} = \mathbf{b}_{in}
\Bigr\},
\end{aligned}
\end{equation}
where $\tilde{\bm{\beta}} = [\bm{\beta}^T, \bm{\gamma}^T]^T$.
\end{proposition}

\begin{proof}
Using the distributive property of Minkowski sum over set difference:
\begin{equation}
\mathcal{R} \oplus \mathcal{CG} = (\mathcal{R}_{\text{outer}} \oplus \mathcal{CG}) \setminus (\mathcal{R}_{\text{inner}} \oplus \mathcal{CG}).
\end{equation}

The outer Minkowski sum:
\begin{equation}
\begin{aligned}
\mathcal{R}_{\text{outer}} \oplus \mathcal{CG} = \{ (\mathbf{c}_o + \mathbf{c}_{\text{ccg}}) + [\mathbf{G}_o, \mathbf{G}_{\text{ccg}}]\tilde{\bm{\beta}} \mid\\ \text{augmented constraints} \}.
\end{aligned}
\end{equation}

The inner Minkowski sum:
\begin{equation}
\begin{aligned}
\mathcal{R}_{\text{inner}} \oplus \mathcal{CG} = \{ (\mathbf{c}_{in} + \mathbf{c}_{\text{ccg}}) + [\mathbf{G}_{in}, \mathbf{G}_{\text{ccg}}]\tilde{\bm{\eta}}\mid \\ \text{augmented constraints} \}.
\end{aligned}
\end{equation}

The exclusion condition requires that points in the outer sum cannot be expressed as elements of the inner sum. Since the CCG component is determined by $\tilde{\bm{\beta}}$, this yields the non-existence constraint in the proposition. The result maintains the RCG structure with augmented generators.
\end{proof}

\subsection{Intersection with Constrained Convex Generators}

The intersection operation between RCG sets and standard CCG sets represents a fundamental operation for compositional analysis. We establish that this operation preserves the RCG structure, demonstrating closure under intersection. 

\begin{proposition}[Intersection of RCG and CCG]
\label{prop:rccg_ccg_intersection}
Let $\mathcal{R}=\langle \mathbf{G}_o,\mathbf{G}_{\mathrm{in}},\mathbf{c}_o,\mathbf{c}_{\mathrm{in}}, \mathcal{P}_o,\mathcal{P}_{\mathrm{in}},\mathbf{r},\mathbf{A}_o,\mathbf{A}_{\mathrm{in}},\mathbf{b}_o,\mathbf{b}_{\mathrm{in}}\rangle$ and $\mathcal{CG}=\langle \mathbf{c}_{\mathrm{ccg}},\mathbf{G}_{\mathrm{ccg}}, \mathcal{J}_{\mathrm{ccg}},\mathcal{P}_{\mathrm{ccg}},\mathbf{A}_{\mathrm{ccg}},\mathbf{b}_{\mathrm{ccg}}\rangle$.
Then their intersection $\mathcal{S}=\mathcal{R}\cap\mathcal{CG}$ admits the unified RCG form
\begin{align}
\mathcal{S}
=\Bigl\{&
\mathbf{c}_o+\tilde{\mathbf{G}}\tilde{\bm{\beta}}
\;\Big|\;
\|\tilde{\bm{\beta}}_{\mathcal{J}_i}\|_{p_i}\le 1,\; i=1,\ldots,k_o,\nonumber \\
&\|\tilde{\bm{\beta}}_{\mathcal{L}_j}\|_{p_j}\le 1,\; j=1,\ldots,m_{\mathrm{ccg}},
\notag\\[-1mm]
&\tilde{\mathbf{A}}\tilde{\bm{\beta}}=\tilde{\mathbf{b}},\;\;
\nexists\,\bm{\eta}:\;
\mathbf{G}_{\mathrm{in}}\bm{\eta}
=(\mathbf{c}_o-\mathbf{c}_{\mathrm{in}})+\mathbf{G}_o\,\Pi_\beta\tilde{\bm{\beta}},
\notag\\
&\|\bm{\eta}_{\mathcal{K}_\ell}\|_{q_\ell}\le r_l,\;\ell=1,\ldots,k_{\mathrm{in}},\;\;
\mathbf{A}_{\mathrm{in}}\bm{\eta}=\mathbf{b}_{\mathrm{in}}
\Bigr\},
\label{eq:rcg_ccg_intersection_unified}
\end{align}
where
\begin{align}
&\tilde{\bm{\beta}}=\begin{bmatrix}\bm{\beta}\\ \bm{\gamma}\end{bmatrix},
\quad
\Pi_\beta=[\,I_p\ \ 0\,],
\quad
\tilde{\mathbf{G}}=[\,\mathbf{G}_o\ \ \mathbf{0}\,],
\notag\\
&\tilde{\mathbf{A}}=
\begin{bmatrix}
\mathbf{A}_o & \mathbf{0}\\
\mathbf{0}   & \mathbf{A}_{\mathrm{ccg}}\\
\mathbf{G}_o & -\mathbf{G}_{\mathrm{ccg}}
\end{bmatrix},
\quad
\tilde{\mathbf{b}}=
\begin{bmatrix}
\mathbf{b}_o\\
\mathbf{b}_{\mathrm{ccg}}\\
\mathbf{c}_{\mathrm{ccg}}-\mathbf{c}_o
\end{bmatrix}.
\end{align}
The index families $\{\mathcal{J}_i\}$ act on the $\bm{\beta}$-block, and $\{\mathcal{L}_j\}$ are the shifted CCG groups acting on the $\bm{\gamma}$-block.
\end{proposition}

\begin{proof}
 Suppose $\mathbf{x}\in\mathcal{R}\cap\mathcal{CG}$. Then there exist $\bm{\beta}$ and $\bm{\gamma}$ such that $\mathbf{x} = \mathbf{c}_o+\mathbf{G}_o\bm{\beta} = \mathbf{c}_{\mathrm{ccg}}+\mathbf{G}_{\mathrm{ccg}}\bm{\gamma}$, with $\|\bm{\beta}_{\mathcal{J}_i}\|_{p_i}\le 1$, $\mathbf{A}_o\bm{\beta}=\mathbf{b}_o$, and $\|\bm{\gamma}_{\mathcal{J}_{\mathrm{ccg},j}}\|_{p_j}\le 1$, $\mathbf{A}_{\mathrm{ccg}}\bm{\gamma}=\mathbf{b}_{\mathrm{ccg}}$. Equating the representations yields the coupling constraint $\mathbf{G}_o\bm{\beta}-\mathbf{G}_{\mathrm{ccg}}\bm{\gamma} = \mathbf{c}_{\mathrm{ccg}}-\mathbf{c}_o$.
By defining the augmented vector $\tilde{\bm{\beta}}=[\bm{\beta}^T, \bm{\gamma}^T]^T$, all linear constraints can be expressed compactly as $\tilde{\mathbf{A}}\tilde{\bm{\beta}}=\tilde{\mathbf{b}}$. The representation of $\mathbf{x}$ becomes $\mathbf{c}_o+\tilde{\mathbf{G}}\tilde{\bm{\beta}}$. Furthermore, since $\mathbf{x}\in\mathcal{R}$, it must satisfy the inner exclusion condition, i.e., there is no $\bm{\eta}$ such that $\mathbf{G}_{\mathrm{in}}\bm{\eta} = (\mathbf{c}_o-\mathbf{c}_{\mathrm{in}})+\mathbf{G}_o\bm{\beta}$ with $\|\bm{\eta}_{\mathcal{K}_\ell}\|_{q_\ell}\le r_l$ and $\mathbf{A}_{\mathrm{in}}\bm{\eta}=\mathbf{b}_{\mathrm{in}}$. Substituting $\bm{\beta} = \Pi_\beta\tilde{\bm{\beta}}$ yields the form in \eqref{eq:rcg_ccg_intersection_unified}.

 Conversely, let $\mathbf{x}$ be a point with the representation \eqref{eq:rcg_ccg_intersection_unified}. The block equation $\tilde{\mathbf{A}}\tilde{\bm{\beta}}=\tilde{\mathbf{b}}$ implies $\mathbf{A}_o\bm{\beta}=\mathbf{b}_o$, $\mathbf{A}_{\mathrm{ccg}}\bm{\gamma}=\mathbf{b}_{\mathrm{ccg}}$, and $\mathbf{G}_o\bm{\beta}-\mathbf{G}_{\mathrm{ccg}}\bm{\gamma} = \mathbf{c}_{\mathrm{ccg}}-\mathbf{c}_o$. The last equality ensures that $\mathbf{x} = \mathbf{c}_o+\mathbf{G}_o\bm{\beta} = \mathbf{c}_{\mathrm{ccg}}+\mathbf{G}_{\mathrm{ccg}}\bm{\gamma}$. The norm and linear constraints on $\bm{\beta}$ and $\bm{\gamma}$ establish that $\mathbf{x}$ belongs to the outer set of $\mathcal{R}$ and also to $\mathcal{CG}$. The non-existence condition on $\bm{\eta}$ ensures that $\mathbf{x}$ is not in the inner set of $\mathcal{R}$. Thus, $\mathbf{x}\in\mathcal{R}\cap\mathcal{CG}$.
\end{proof}

\begin{remark}[Minkowski Sum and Intersection of RCG with RCG]
\label{rem:rcg_rcg_intersection}
The Minkowski sum and intersection of two RCG sets presents a more complex scenario that is not addressed in Propositions~\ref{prop:rccg_minkowski_sum} and~\ref{prop:rccg_ccg_intersection}. Unlike the RCG-CCG Minkowski sum and intersection, the RCG-RCG Minkowski sum and intersection do not maintain closure within the RCG class. Specifically, the Minkowski sum and intersection of two RCG sets, each containing a single topological hole, may yield a set with multiple disjoint or joint holes. This occurs because the exclusion regions from different RCGs can create separate voids in the resulting feasible region. Formally representing such multiply-connected domains requires extending beyond the single-hole RCG framework presented in this work. Future research will address this challenge through hybrid RCG formulations, which will systematically handle sets with multiple topological holes.
\end{remark}

The resulting representation in \ref{prop:rccg_ccg_intersection} matches the standard RCG unified form with augmented generators, centers, and constraints. Therefore, $\mathcal{S} = \mathcal{R} \cap \mathcal{CG}$ is itself an RCG, confirming closure under intersection operations. This closure property enables compositional analysis techniques where complex feasible regions are constructed through successive intersections while maintaining computational tractability within the RCG framework.

\section{Special Cases Based on Norm Selection}

The general RCG framework encompasses several important special cases that arise from specific choices of norms in the parameter constraints. These special cases offer computational advantages and geometric interpretations that make them particularly suitable for different application domains. We formally define three primary specializations that frequently appear in control and verification problems.

\begin{definition}[Roundabout Ellipsotopes]
\label{def:roundabout_ellipsotopes}
A \emph{Roundabout Ellipsotope} $\mathcal{RE}$ is an RCG where all norm constraints employ the Euclidean norm ($p_i = q_j = 2$ for all $i,j$). The set takes the form:
\begin{equation}
\mathcal{RE} = \mathcal{E}_{\text{outer}} \setminus \mathcal{E}_{\text{inner}},
\end{equation}
where both $\mathcal{E}_{\text{outer}}$ and $\mathcal{E}_{\text{inner}}$ are ellipsotopes~\cite{kousik2022ellipsotopes} with $L_2$-norm constraints on their respective parameter vectors. The tuple representation simplifies to:
\begin{equation}
\mathcal{RE} = \langle \mathbf{G}_o, \mathbf{G}_{in}, \mathbf{c}_o, \mathbf{c}_{in}, \mathcal{J}_o, \mathcal{J}_{in}, \mathbf{r}, \mathbf{A}_o, \mathbf{A}_{in}, \mathbf{b}_o, \mathbf{b}_{in} \rangle_{2}.
\end{equation}
where the norm specifications are implicitly $p_i = q_j = 2$ throughout.
\end{definition}

\begin{definition}[Roundabout Constrained Zonotopes]
\label{def:roundabout_constrained_zonotopes}
A \emph{Roundabout Constrained Zonotope} $\mathcal{RCZ}$ is an RCG where all norm constraints employ the infinity norm ($p_i = q_j = \infty$ for all $i,j$). The set is characterized by:
\begin{equation}
\mathcal{RCZ} = \mathcal{Z}_{\text{outer}}^c \setminus \mathcal{Z}_{\text{inner}}^c,
\end{equation}
where $\mathcal{Z}_{\text{outer}}^c$ and $\mathcal{Z}_{\text{inner}}^c$ are constrained zonotopes with box constraints $\|\bm{\beta}\|_\infty \leq 1$ and $\|\bm{\eta}\|_\infty \leq r_j$ respectively, along with linear equality constraints. The tuple representation becomes:
\begin{equation}
\mathcal{RCZ} = \langle \mathbf{G}_o, \mathbf{G}_{in}, \mathbf{c}_o, \mathbf{c}_{in}, \mathbf{r}, \mathbf{A}_o, \mathbf{A}_{in}, \mathbf{b}_o, \mathbf{b}_{in} \rangle_{\infty}.
\end{equation}
with the understanding that all partitions use infinity norm constraints.
\end{definition}

\begin{definition}[Roundabout Zonotopes]
\label{def:roundabout_zonotopes}
A \emph{Roundabout Zonotope} $\mathcal{RZ}$ is a special case of $\mathcal{RCZ}$ where no additional linear equality constraints are imposed ($\mathbf{A}_o = \mathbf{A}_{in} = \emptyset$). The set reduces to:
\begin{equation}
    \mathcal{RZ} \;=\; \mathcal{Z}_{\text{outer}} \setminus \mathcal{Z}_{\text{inner}},
\end{equation}
where
\begin{align}
    \mathcal{Z}_{\text{outer}} 
        &= \big\{ \mathbf{c}_o + \mathbf{G}_o \bm{\beta} \;\big|\; \|\bm{\beta}\|_\infty \le 1 \big\},\\
    \mathcal{Z}_{\text{inner}} 
        &= \big\{ \mathbf{c}_{in} + \tilde{\mathbf{G}}_{in} \bm{\eta} \;\big|\; \|\bm{\eta}\|_\infty \le 1 \big\},
\end{align}
and the inner generator matrix absorbs the scale factor, 
\begin{equation}
    \tilde{\mathbf{G}}_{in} \;\coloneqq\; \mathbf{r}^T\,\mathbf{G}_{in},\qquad\forall r \in \mathbf{r}, \qquad 0 \le r < 1.
\end{equation}
With this normalization, the tuple representation becomes
\begin{equation}
    \mathcal{RZ} \;=\; \langle \mathbf{G}_o, \tilde{\mathbf{G}}_{in}, \mathbf{c}_o, \mathbf{c}_{in},\mathbf{r} \rangle_{\infty}.
\end{equation}
\end{definition}

When the generator matrices of both CCGs coincide, the roundabout constrained convex generators admits a simplified representation.

\begin{corollary}[Common Generator RCG]
\label{cor:common_generator}
When $\mathbf{G}_o = \mathbf{G}_{in} = \mathbf{G}$, the roundabout constrained convex generators becomes:
\begin{equation}
\begin{aligned}
\mathcal{R} = \Bigl\{ &\mathbf{c}_o + \mathbf{G}\bm{\beta} \;\Big|\; \bm{\beta} \in \mathcal{C}_o(\mathcal{P}_o, \mathbf{A}_o, \mathbf{b}_o), \\
&\nexists \bm{\eta} \in \mathcal{C}_{in}(\mathcal{P}_{in}, \mathbf{A}_{in}, \mathbf{b}_{in}): \mathbf{G}\bm{\eta} = (\mathbf{c}_o - \mathbf{c}_{in}) + \mathbf{G}\bm{\beta}
\Bigr\},
\end{aligned}
\end{equation}
where $\mathcal{C}_o$ and $\mathcal{C}_{in}$ denote the constraint sets for the outer and inner parameters respectively. The shared generator matrix $\mathbf{G}$ simplifies the exclusion condition while maintaining the center offset.
\end{corollary}

When the centers coincide, the transformation mapping eliminates the translation term.

\begin{corollary}[Concentric RCG]
\label{cor:concentric}
For concentric CCGs with $\mathbf{c}_o = \mathbf{c}_{in} = \mathbf{c}$, the unified representation reduces to:
\begin{equation}
\begin{aligned}
\mathcal{R} = \Bigl\{ &\mathbf{c} + \mathbf{G}_o\bm{\beta} \;\Big|\; \bm{\beta} \in \mathcal{C}_o(\mathcal{P}_o, \mathbf{A}_o, \mathbf{b}_o), \\
&\nexists \bm{\eta} \in \mathcal{C}_{in}(\mathcal{P}_{in}, \mathbf{A}_{in}, \mathbf{b}_{in}): \mathbf{G}_{in}\bm{\eta} = \mathbf{G}_o\bm{\beta}
\Bigr\}.
\end{aligned}
\end{equation}
The absence of the center offset term yields a purely generator-based mapping between the parameter spaces.
\end{corollary}

The most significant simplification occurs when both the center and the generator matrices coincide.

\begin{corollary}[Common Center and Generator RCG]
\label{cor:common_center_generator}
When $\mathbf{c}_o = \mathbf{c}_{in} = \mathbf{c}$ and $\mathbf{G}_o = \mathbf{G}_{in} = \mathbf{G}$, the roundabout constrained convex generators reduces to:
\begin{equation}
\begin{aligned}
\mathcal{R} = \Bigl\{ &\mathbf{c} + \mathbf{G}\bm{\beta} \;\Big|\; \\
&\bm{\beta} \in \mathcal{C}_o(\mathcal{P}_o, \mathbf{A}_o, \mathbf{b}_o), \\
&\nexists \bm{\eta} \in \mathcal{C}_{in}(\mathcal{P}_{in}, \mathbf{A}_{in}, \mathbf{b}_{in}): \mathbf{G}\bm{\eta} = \mathbf{G}\bm{\beta}
\Bigr\}.
\end{aligned}
\end{equation}

When $\mathbf{G}$ has full column rank (i.e., $\ker(\mathbf{G}) = \{0\}$), which commonly occurs when the number of generators is large, the equation $\mathbf{G}\bm{\eta} = \mathbf{G}\bm{\beta}$ has at most one solution: $\bm{\eta} = \bm{\beta}$. In this case, for the special instance of RCGs (no constraints), the representation simplifies to:
\begin{equation}
\mathcal{R} = \{ \mathbf{c} + \mathbf{G}\bm{\beta} \mid r_{_i} \leq \|\bm{\beta_{\mathcal{J}_i}}\|_{p_i} \leq 1 \}.
\end{equation}
This yields an exact parameter space representation where the annular structure is directly encoded through the norm bounds.
\end{corollary}

These special cases arise frequently in applications where symmetry or structural regularity is present. The common generator case often appears in uncertainty propagation where the same basis directions are used with different magnitudes, while the concentric case naturally emerges in tolerance analysis and safety verification around equilibrium points. The common center and generator case provides the most computationally efficient representation, particularly when full column rank allows for direct parameter space characterization without requiring iterative feasibility checks.

\subsection{Illustrative Example}

To demonstrate the flexibility of  RCGs, we present an example where both the outer and inner boundaries follow a $2$-norm, creating concentric elliptical shapes by using CORA~\cite{althoff2015introduction} in MATLAB.

\begin{figure}[h]
\centering
\includegraphics[width=0.49\textwidth]{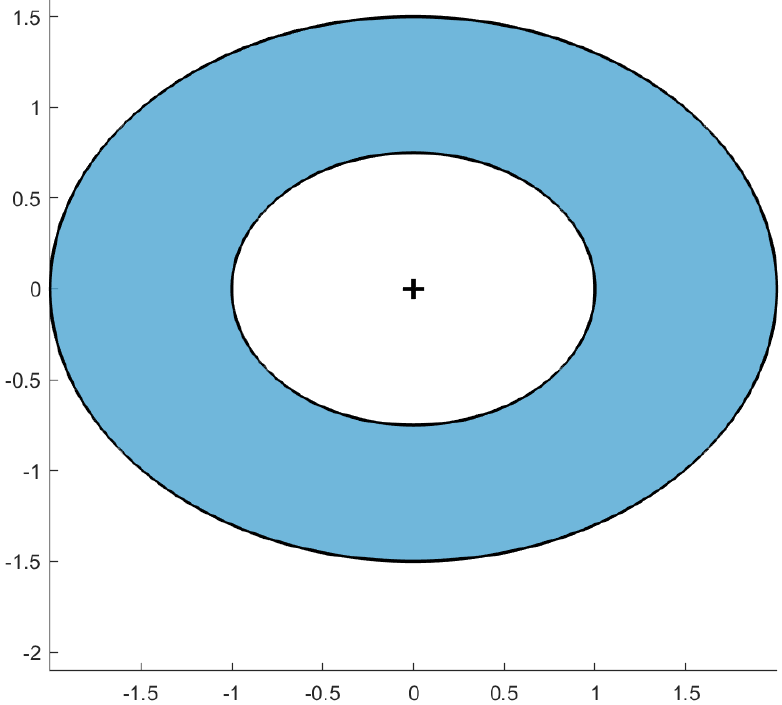}
\caption{A $\mathcal{RE}$ with elliptical outer boundary ($p = 2$) and elliptical inner boundary ($p = 2$). The blue region represents the feasible set $\mathcal{RE}$.}
\label{fig:re}
\end{figure}

\begin{example}[$\mathcal{RE}$ with Common Center and Common Generators]
\label{ex:elliptical_donut}
Consider a  RCG with the following parameters:
\begin{align}
\mathbf{c}_{\text{out}} = \mathbf{c}_{\text{in}} &= \begin{bmatrix} 0 \\ 0 \end{bmatrix},\quad  \mathbf{G}_{\text{out}} = \mathbf{G}_{\text{in}} = \begin{bmatrix} 2 & 0 \\ 0 & 1.5 \end{bmatrix}, \nonumber\\
\mathcal{J}_{\text{out}} = \mathcal{J}_{\text{in}} &= \{1, 2\},\quad 
p_{\text{out}} = p_{\text{in}} = 2, \quad 
\mathbf{r} = 0.5.
\end{align}

This creates a $\mathcal{RE}$ where:
\begin{itemize}
\item The outer constraint is defined by $\|\bm{\beta}\|_{2} \leq 1$
\item The inner constraint is defined by $0.5 \leq \|\bm{\beta}\|_2 $
\end{itemize}
\end{example}

Figure~\ref{fig:re} illustrates this example.

\begin{example}[Elliptical Outer and Polygonal Inner  RCG]
\label{ex:square_circular_donut}
Consider a  RCG with the following parameters:
\begin{align}
\mathbf{c}_{\text{out}} &= \begin{bmatrix} 0 \\ 0 \end{bmatrix}, 
\mathbf{c}_{\text{in}} = \begin{bmatrix} 1.2 \\ 1 \end{bmatrix},
\mathbf{G}_{\text{out}} =
\mathbf{G}_{\text{in}} = \begin{bmatrix} 3 & 0 & 1 \\ 0 & 2 & 1 \end{bmatrix}, \nonumber\\
\mathcal{J}_{\text{out}} &= \{1:3\},
\mathcal{J}_{\text{in}} = \{1:3\}, p_{\text{out}} =  \infty,
p_{\text{in}} =2, r = 0.2, \nonumber\\
A_{\text{out}} &= \begin{bmatrix} 1 & 0.5 & 1 \end{bmatrix}, 
A_{\text{in}} = \begin{bmatrix} 1 & 0 & 1 \end{bmatrix}, b_{\text{out}} = b_{\text{in}} = 1. \quad
\end{align}

This creates a set with:
\begin{itemize}
    \item A zonotopic CCG defined by $\|\bm{\beta}\|_{\infty} \leq 1$.
    \item A elliptical CCG defined by $\|\bm{\beta}\|_2 \leq 0.2$.
\end{itemize}
\end{example}

\begin{figure}[h]
\centering
\includegraphics[width=0.45\textwidth]{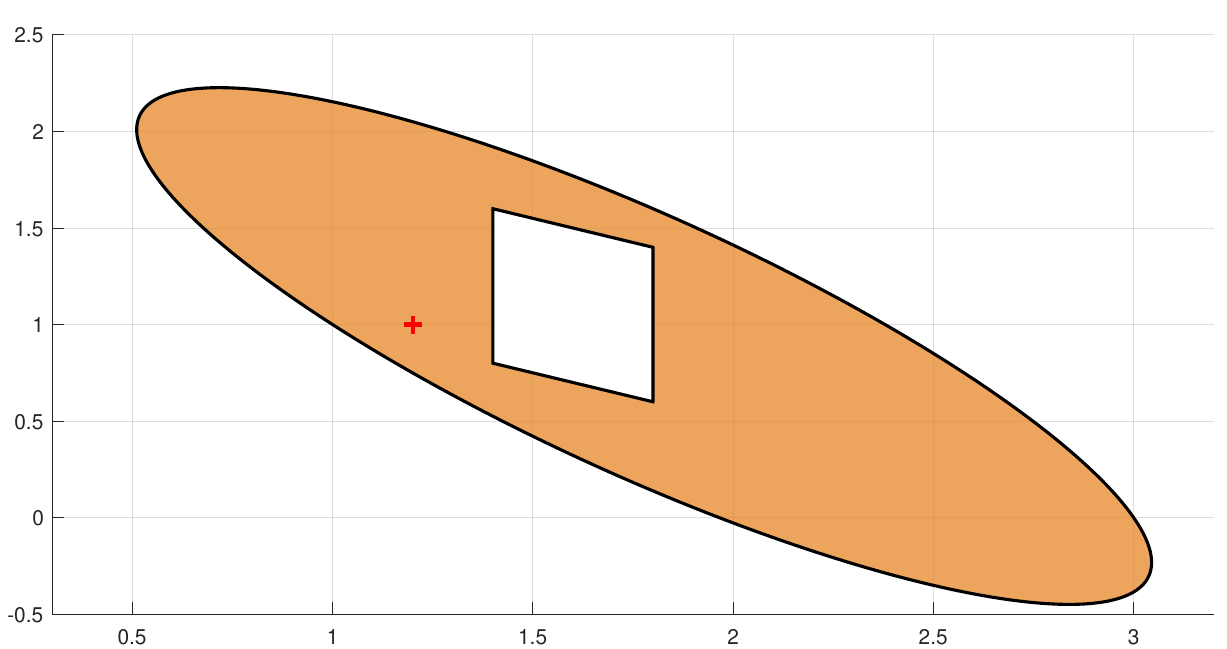}
\caption{RCG with elliptical outer boundary ($p = 2$) and polygonal inner boundary ($p = \infty$). The orange region represents the feasible set $\mathcal{R}$.}
\label{fig:square_circle_donut}
\end{figure}

Figure~\ref{fig:square_circle_donut} illustrates this example, showing how different $p$-norms create distinct geometric shapes for the inner and outer boundaries.

\begin{figure}[htbp]
\centering
\includegraphics[width=0.48\textwidth]{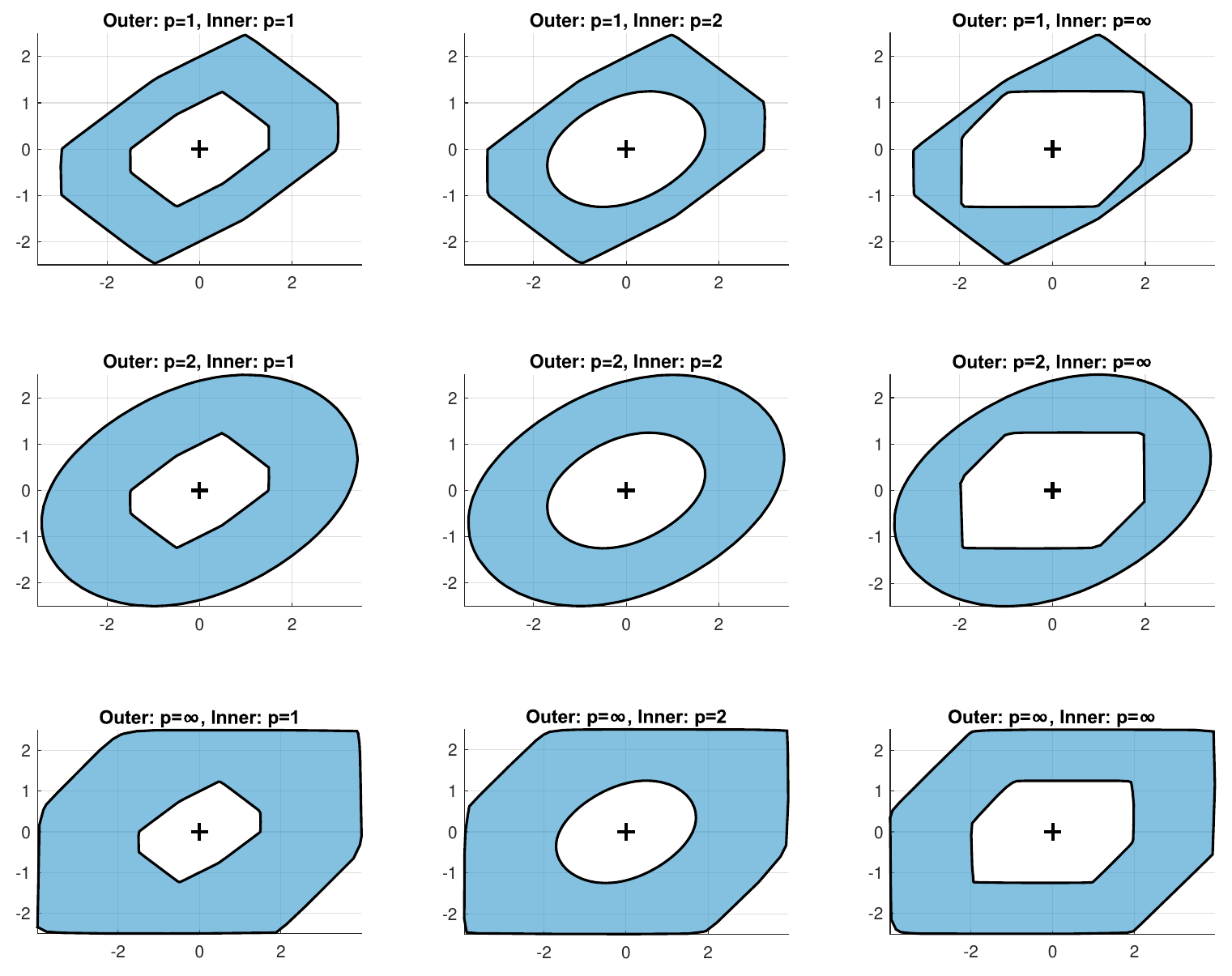}
\caption{Nine common  RCG configurations obtained by varying outer and inner $p$-norms. Columns correspond to outer boundary norms ($p = 1, 2, \infty$) and rows to inner boundary norms ($p = 1, 2, \infty$).}
\label{fig:donut_CCGs_grid}
\end{figure}

Figure~\ref{fig:donut_CCGs_grid} illustrates the geometric versatility of  RCGs through systematic variation of the $p$-norm parameters. The $3 \times 3$ grid demonstrates all combinations of three fundamental norm choices with common center and common generators: $p = 1$, $p = 2$ (ellipse), and $p = \infty$ (polygon).



Based on the general framework established in Proposition~\ref{prop:rccg_ccg_intersection}, we can derive the intersection formula for roundabout zonotopes as a special case where all norm constraints employ the infinity norm and no additional linear constraints are present beyond those required for intersection.

\begin{corollary}[Intersection of Roundabout Zonotope with Zonotope]
\label{cor:rz_intersection}
Consider a roundabout zonotope $\mathcal{RZ} = \langle \mathbf{G}_o, \mathbf{G}_{in}, \mathbf{c}_o, \mathbf{c}_{in}, \mathbf{r} \rangle$ and a zonotope $\mathcal{Y} = \{\mathbf{c}_y + \mathbf{G}_y\bm{\gamma} \mid \|\bm{\gamma}\|_\infty \leq 1\}$. Their intersection yields:
\begin{equation}
\begin{aligned}
\mathcal{RZ} \cap \mathcal{Y} = \Bigl\{ &\mathbf{c}_o + \tilde{\mathbf{G}}\tilde{\bm{\beta}} \;\Big|\;\|\tilde{\bm{\beta}}\|_\infty \leq 1,\\
&[\mathbf{G}_o, -\mathbf{G}_y]\tilde{\bm{\beta}} = \mathbf{c}_y - \mathbf{c}_o, \\
&\nexists \bm{\eta}: \mathbf{G}_{in}\bm{\eta} = (\mathbf{c}_o - \mathbf{c}_{in}) + \mathbf{G}_o\bm{\beta}, \|\bm{\eta}\|_\infty \leq r_j
\Bigr\},
\end{aligned}
\end{equation}

where $\tilde{\bm{\beta}} = [\bm{\beta}^T, \bm{\gamma}^T]^T$ is the augmented parameter vector, $\tilde{\mathbf{G}} = [\mathbf{G}_o, \mathbf{0}]$ is the augmented generator matrix.
\end{corollary}

\begin{figure*}[h!]
\centering
\includegraphics[width=0.96\textwidth]{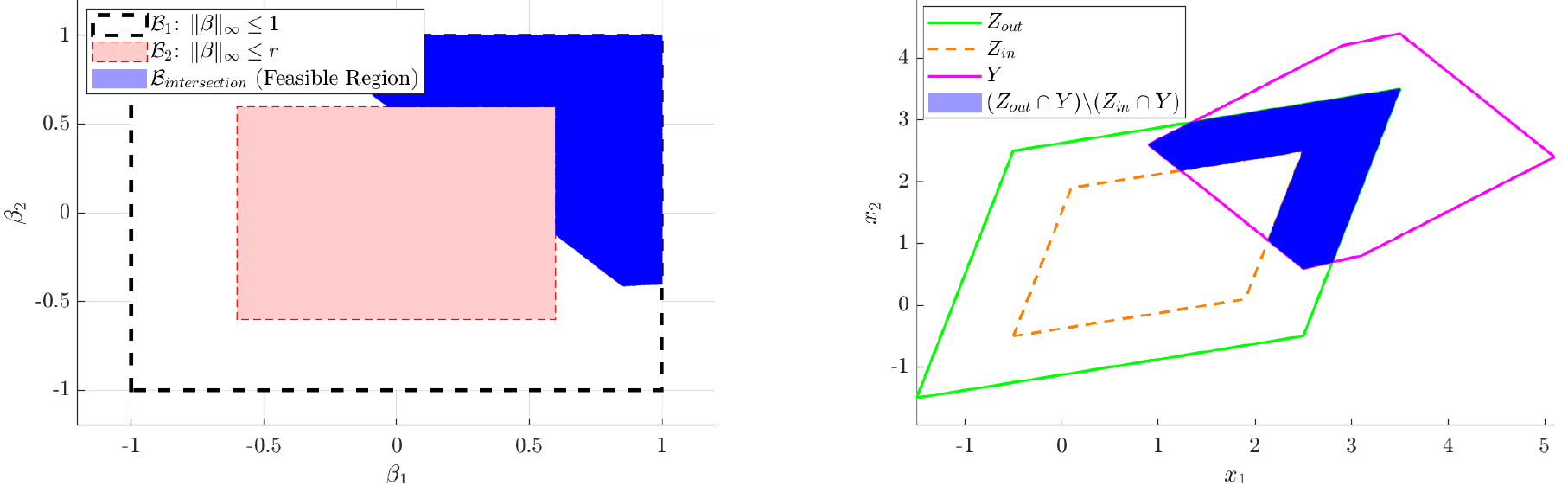}
\caption{Visualization of the roundabout zonotope intersection. Left: Parameter space $(\beta_1, \beta_2)$ showing the feasible region $\mathcal{B}_{\text{intersection}}$ (blue) as the intersection of the outer constraint $\mathcal{B}_1: \|\bm{\beta}\|_\infty \leq 1$ (black dashed box) with the constraint from $\mathcal{Y}$, while excluding the inner region $\mathcal{B}_2: \|\bm{\beta}\|_\infty \leq r$ (red dashed box). Right: The resulting geometric sets in $\mathbb{R}^2$ showing $\mathcal{Z}_{\text{out}}$ (green), $\mathcal{Z}_{\text{in}}$ (orange dashed), $\mathcal{Y}$ (magenta), and the final intersection $(\mathcal{Z}_{\text{out}} \cap \mathcal{Y}) \setminus (\mathcal{Z}_{\text{in}} \cap \mathcal{Y})$ (blue filled region).}
\label{fig:intersection_example}
\end{figure*}

\begin{remark}[Concentric Case with Shared Generators]
\label{rem:concentric_intersection}
When the outer and inner zonotopes share the same center and generator matrix ($\mathbf{c}_1 = \mathbf{c}_2 = \mathbf{c}$ and $\mathbf{G}_1 = \mathbf{G}_2 = \mathbf{G}$), the intersection formula simplifies considerably. In this case:

\begin{itemize}
\item The transformation mapping reduces to $\mathbf{h}(\tilde{\bm{\beta}}) = (\mathbf{c} - \mathbf{c}) + \mathbf{G}\bm{\beta} = \mathbf{G}\bm{\beta}$
\item The inner exclusion condition becomes: $\nexists \bm{\eta}: \mathbf{G}\bm{\eta} = \mathbf{G}\bm{\beta}$ with $\|\bm{\eta}\|_\infty \leq r$
\end{itemize}


Furthermore, when $\mathbf{G}$ has full column rank, we mention in ~\ref{cor:common_center_generator}. The condition reduces to simply $\|\bm{\beta}\|_\infty > r$, yielding the elegant form:
\begin{equation}
\begin{aligned}
\mathcal{RZ}\cap \mathcal{Y} = \Bigl\{ \mathbf{x} = \mathbf{c} + \mathbf{G}\bm{\beta} \;\Big|\; r < \|\bm{\beta}\|_\infty \leq 1, \\\exists \bm{\gamma}: [\mathbf{G}, -\mathbf{G}_y]\begin{bmatrix}\bm{\beta} \\ \bm{\gamma}\end{bmatrix} = \mathbf{c}_y - \mathbf{c}, \;\|\bm{\gamma}\|_\infty \leq 1 \Bigr\}.
\end{aligned}
\end{equation}
\end{remark}

\begin{example}[Numerical Illustration of Intersection]
\label{ex:intersection_numerical}
To illustrate the theoretical framework developed above, we present a concrete numerical example demonstrating the intersection of a ring-shaped roundabout zonotope with another zonotope. This example showcases both the parameter space constraints and the resulting geometric sets.

Consider a roundabout zonotope $\mathcal{RZ} = \mathcal{Z}_{\text{out}} \setminus \mathcal{Z}_{\text{in}}$ where both the outer and inner zonotopes share the same center and the same generator matrix (this case discussed in Corollary~\ref{cor:common_center_generator}). The system parameters are:

\begin{equation}
\mathbf{c} = \begin{bmatrix} 1 \\ 1 \end{bmatrix}, \quad 
\mathbf{G} = \begin{bmatrix} 2 & 0.5 \\ 0.5 & 2 \end{bmatrix}, \quad
r = 0.6.
\end{equation}
The second zonotope $\mathcal{Y}$ is defined with:
\begin{equation}
\mathbf{c}_y = \begin{bmatrix} 3 \\ 2.5 \end{bmatrix}, \quad
\mathbf{G}_y = \begin{bmatrix} 1 & -0.8 & 0.3 \\ 0.8 & 1 & 0.1 \end{bmatrix}.
\end{equation}

For this configuration, the feasible parameter set becomes:
\begin{equation}
\begin{aligned}
\mathcal{B}_{\text{intersection}} = \Bigl\{ \bm{\beta} \in \mathbb{R}^2 \;\Big|\; &0.6 < \|\bm{\beta}\|_\infty \leq 1, \\
&\exists \bm{\gamma} \in \mathbb{R}^3: \mathbf{G}\bm{\beta} - \mathbf{G}_y\bm{\gamma} = \mathbf{c}_y - \mathbf{c}, \\
&\|\bm{\gamma}\|_\infty \leq 1 \Bigr\}.
\end{aligned}
\end{equation}
The linear constraint $\mathbf{G}\bm{\beta} - \mathbf{G}_y\bm{\gamma} = \mathbf{c}_y - \mathbf{c}$ expands to:
\begin{equation}
\begin{bmatrix} 2 & 0.5 \\ 0.5 & 2 \end{bmatrix} \begin{bmatrix} \beta_1 \\ \beta_2 \end{bmatrix} - 
\begin{bmatrix} 1 & -0.8 & 0.3 \\ 0.8 & 1 & 0.1 \end{bmatrix} \begin{bmatrix} \gamma_1 \\ \gamma_2 \\ \gamma_3 \end{bmatrix} = 
\begin{bmatrix} 2 \\ 1.5 \end{bmatrix}.
\end{equation}

This system determines which values of $\bm{\beta}$ can be matched with valid $\bm{\gamma}$ parameters, effectively constraining the feasible region.

As illustrated in Figure~\ref{fig:intersection_example}, the parameter space visualization (left panel) reveals how the intersection constraint carves out a non-convex feasible region. The blue shaded area represents parameters $\bm{\beta}$ that simultaneously satisfy three conditions: they lie within the outer box constraint $\|\bm{\beta}\|_\infty \leq 1$, they avoid the inner exclusion region $\|\bm{\beta}\|_\infty \leq 0.6$, and they permit a valid correspondence with the zonotope $\mathcal{Y}$ through the linear constraint system.

The geometric representation (right panel) demonstrates the actual sets in $\mathbb{R}^2$. The outer zonotope $\mathcal{Z}_{\text{out}}$ (green boundary) and inner zonotope $\mathcal{Z}_{\text{in}}$ (orange dashed boundary) form the ring structure, while $\mathcal{Y}$ (magenta boundary) intersects this ring to produce the final blue region. This region maintains the roundabout zonotope structure, confirming our theoretical closure result ~\ref{prop:rccg_ccg_intersection}.

\end{example}

\section{Conclusion}
This paper introduced the Roundabout Constrained Convex Generators (RCGs) framework for representing regions with a single topological hole in control applications. The current work focuses specifically on regions containing exactly one hole; extensions to multiple or overlapping holes will be addressed through hybrid RCGs formulations in future work.
We established that RCG sets remain closed under linear transformations, Minkowski sums with CCGs, and intersections with CCGs. These closure properties provide the theoretical foundation for control synthesis within the RCG framework. Special cases including roundabout zonotopes and ellipsotopes offer computational and presentation advantages.

While this paper establishes the mathematical foundations of RCGs, implementation of reachability analysis algorithms remains future work. Next steps include developing efficient computational procedures for RCG operations, extending to hybrid formulations for multiple exclusion regions, and demonstrating applications in collision avoidance and robust control scenarios.

\bibliographystyle{IEEEtran} 
\bibliography{BibTex_2021}

@inproceedings{girard2005reachability,
  title     = {Reachability of Uncertain Linear Systems Using Zonotopes},
  author    = {Girard, Antoine},
  booktitle = {Hybrid Systems: Computation and Control (HSCC)},
  series    = {LNCS},
  volume    = {3414},
  pages     = {291--305},
  year      = {2005},
  publisher = {Springer},
  doi       = {10.1007/978-3-540-31954-2_19}
}

@article{althoff2010reachability,
  title   = {Reachability Analysis of Linear Systems with Uncertain Parameters and Inputs},
  author  = {Althoff, Matthias and Krogh, Bruce H.},
  journal = {IEEE Transactions on Automatic Control},
  volume  = {55},
  number  = {5},
  pages   = {1157--1164},
  year    = {2010},
  doi     = {10.1109/TAC.2010.2042078}
}

@inproceedings{alanwar2022data,
  title={Data-driven set-based estimation using matrix zonotopes with set containment guarantees},
  author={Alanwar, Amr and Berndt, Alexander and Johansson, Karl Henrik and Sandberg, Henrik},
  booktitle={2022 European Control Conference (ECC)},
  pages={875--881},
  year={2022},
  organization={IEEE}
}

@article{scott2016constrained,
  title   = {Constrained Zonotopes: A New Tool for Set-Based Estimation and Fault Detection},
  author  = {Scott, Joseph K. and Raimondo, Davide M. and Marseglia, Giuseppe R. and Braatz, Richard D.},
  journal = {Automatica},
  volume  = {69},
  pages   = {126--136},
  year    = {2016},
  doi     = {10.1016/j.automatica.2016.02.036}
}

@article{bird2023hybrid,
  title={Hybrid zonotopes: A new set representation for reachability analysis of mixed logical dynamical systems},
  author={Bird, Trevor J and Pangborn, Herschel C and Jain, Neera and Koeln, Justin P},
  journal={Automatica},
  volume={154},
  pages={111107},
  year={2023},
  publisher={Elsevier}
}

@inproceedings{althoff2013reachability,
  title={Reachability analysis of nonlinear systems using conservative polynomialization and non-convex sets},
  author={Althoff, Matthias},
  booktitle={Proceedings of the 16th international conference on Hybrid systems: computation and control},
  pages={173--182},
  year={2013}
}

@article{bird2021unions,
  title={Unions and complements of hybrid zonotopes},
  author={Bird, Trevor J and Jain, Neera},
  journal={IEEE Control Systems Letters},
  volume={6},
  pages={1778--1783},
  year={2021},
  publisher={IEEE}
}

@article{xie2025data,
  title={Data-driven reachability analysis for piecewise affine systems},
  author={Xie, Peng and Betz, Johannes and Raimondo, Davide M and Alanwar, Amr},
  journal={arXiv preprint arXiv:2504.04362},
  year={2025}
}

@article{xie2025hybrid,
  title={Hybrid Polynomial Zonotopes: A Set Representation for Reachability Analysis in Hybrid Nonaffine Systems},
  author={Xie, Peng and Zhang, Zhen and Alanwar, Amr},
  journal={arXiv preprint arXiv:2506.13567},
  year={2025}
}

@article{silverccg2022,
  title={Constrained convex generators: A tool suitable for set-based estimation with range and bearing measurements},
  author={Silvestre, Daniel},
  journal={IEEE Control Systems Letters},
  volume={6},
  pages={1610--1615},
  year={2021},
  publisher={IEEE}
}

@article{kochdumper2023constrained,
  title={Constrained polynomial zonotopes},
  author={Kochdumper, Niklas and Althoff, Matthias},
  journal={Acta Informatica},
  volume={60},
  number={3},
  pages={279--316},
  year={2023},
  publisher={Springer}
}

@article{kousik2022ellipsotopes,
  title={Ellipsotopes: Uniting ellipsoids and zonotopes for reachability analysis and fault detection},
  author={Kousik, Shreyas and Dai, Adam and Gao, Grace Xingxin},
  journal={IEEE Transactions on Automatic Control},
  volume={68},
  number={6},
  pages={3440--3452},
  year={2022},
  publisher={IEEE}
}

@article{walsh1956conformal,
  title={On the conformal mapping of multiply connected regions},
  author={Walsh, JL},
  journal={Transactions of the American Mathematical Society},
  volume={82},
  number={1},
  pages={128--146},
  year={1956}
}

@article{silvestre2023exact,
  title={Exact set-valued estimation using constrained convex generators for uncertain linear systems},
  author={Silvestre, Daniel},
  journal={IFAC-PapersOnLine},
  volume={56},
  number={2},
  pages={9461--9466},
  year={2023},
  publisher={Elsevier}
}

@article{alanwar2023data,
  title={Data-driven reachability analysis from noisy data},
  author={Alanwar, Amr and Koch, Anne and Allg{\"o}wer, Frank and Johansson, Karl Henrik},
  journal={IEEE Transactions on Automatic Control},
  volume={68},
  number={5},
  pages={3054--3069},
  year={2023},
  publisher={IEEE}
}

@article{xie2025informed,
  title={Informed Hybrid Zonotope-based Motion Planning Algorithm},
  author={Xie, Peng and Betz, Johannes and Alanwar, Amr},
  journal={arXiv preprint arXiv:2507.09309},
  year={2025}
}

@inproceedings{althoff2015introduction,
  title={An introduction to CORA 2015},
  author={Althoff, Matthias},
  booktitle={Proc. of the workshop on applied verification for continuous and hybrid systems},
  pages={120--151},
  year={2015}
}

@article{ghrist2008barcodes,
  title={Barcodes: the persistent topology of data},
  author={Ghrist, Robert},
  journal={Bulletin of the American Mathematical Society},
  volume={45},
  number={1},
  pages={61--75},
  year={2008}
}

\end{document}